\documentclass[12pt,a4paper,reqno,twoside]{amsart}

\usepackage[english]{babel}
\usepackage{stmaryrd}
\usepackage{dsfont}
\usepackage[symbol*,ragged]{footmisc}
\usepackage[colorlinks,linkcolor=red,anchorcolor=blue,citecolor=blue,urlcolor=blue]{hyperref}
\usepackage{color,xcolor}

\usepackage{geometry}
\usepackage{amssymb}
\usepackage{amsmath}
\usepackage{mathrsfs}
\usepackage{amsfonts}
\usepackage{epsfig}

\usepackage{amsthm}
\usepackage{amsxtra}
\usepackage{bbding}
\usepackage{epsfig}
\usepackage{graphicx}
\usepackage{latexsym}
\usepackage{mathbbol}
\usepackage{bbold}

\usepackage{pifont}
\usepackage{wasysym}
\usepackage{skull}
\usepackage{float}

\DeclareSymbolFontAlphabet{\mathbb}{AMSb}
\DeclareSymbolFontAlphabet{\mathbbol}{bbold}

\usepackage{amscd}
\usepackage[all]{xy}
\allowdisplaybreaks[4]
\usepackage{setspace}

\theoremstyle{plain}

\newtheorem*{theorem*}{\normalfont\scshape Theorem}
\newtheorem{proposition}{\normalfont\scshape Proposition}[section]
\newtheorem{lemma}[proposition]{\normalfont\scshape Lemma}

\newtheorem*{corollary*}{\normalfont\scshape Corollary}

\theoremstyle{remark}
\newtheorem*{remark*}{\normalfont\scshape Remark}

\numberwithin{equation}{section}
\addtocounter{footnote}{1}

\renewcommand{\footnoterule}{
  \kern -3pt
  \hrule width 2.5in height 0.4pt
  \kern 3pt
}

\makeatletter
\@ifundefined{MakeUppercase}{}{}
\makeatother

\begin{document}
	
\title[ On a sum of the error term of the Dirichlet divisor function over primes ]
	  { On a sum of the error term of the Dirichlet divisor function over primes }

\author[Zhen Guo, Xin Li]{Zhen Guo \quad \& \quad Xin Li}

\address{Department of Mathematics, China University of Mining and Technology,
         Beijing, 100083, People's Republic of China}

\email{zhen.guo.math@gmail.com}


\address{Department of Mathematics, China University of Mining and Technology,
         Beijing, 100083, People's Republic of China}

\email{lixin\_alice@foxmail.com}

\date{}

\footnotetext[1]{Zhen Guo is the corresponding author. 

\quad\,\,
{\textbf{MR(2020) Subject Classification}}: 11L07, 11L20, 11N37.

}

\begin{abstract}
 Let $d(n)$ be the Dirichlet divisor function and $\Delta(x)$ denote the error term of the sum $\sum_{n\leqslant x}d(n)$ for a large real variable $x$. In this paper we focus on the sum $\sum_{p\leqslant x}\Delta^2(p)$, where $p$ runs over primes. We prove that there exists an asymptotic formula.
\end{abstract}

\maketitle

\section{Introduction and Main Result}
 Let $d(n)$ be the Dirichlet divisor function and $\Delta(x)$ denote the error term of the sum $\sum_{n\leqslant x}d(n)$ for a large real variable $x$. It was first proved by Dirichlet that $\Delta(x)=O(x^{1/2})$. Over the years the exponent $1/2$ was improved by many authors\cite{MR1199067,MR0644943,MR1512430,MR1509041}. Until now the best result
\begin{equation}
  \Delta(x)\ll x^{131/416}(\log x)^{26497/8320}
\end{equation}
was given by Huxley\cite{MR2005876}.

It is conjectured that $\Delta(x)\ll x^{1/4+\varepsilon}$, which is supported by the classical mean square result
\begin{equation}\label{mean square}
  \int_{1}^{T}\Delta^2(x)dx=\frac{\mathcal{C}}{6\pi^2}T^{3/2}+O(T^{5/4+\varepsilon})
\end{equation}
for a large real number $T$, and where 
\begin{equation*}
  \mathcal{C}=\sum_{n=1}^{\infty}\frac{d^2(n)}{n^{3/2}}
\end{equation*}
is a constant. It was proved by Cram\'{e}r\cite{MR1544568} in 1922. The result was improved by Tong\cite{MR0098718} in 1956, Preissmann\cite{MR0930552} in 1988, and Lau and Tsang\cite{MR2475967} in 2009.

Also, the discrete mean values
\begin{equation*}
  \mathcal{D}_2(x):=\sum_{n\leqslant x}\Delta^2(n)
\end{equation*}
have been investigated by various authors. Define the continuous mean values
\begin{equation*}
  \mathcal{C}_2(x):=\int_{1}^{x}\Delta^2(t)dt.
\end{equation*}
Many authors found that the discrete and the continuous mean value formulas are connected deeply with each other and studied the difference between $\mathcal{D}_2(x)$ and $\mathcal{C}_2(x)$.  Hardy\cite{MR1576556} studied the difference between $\mathcal{D}_2(x)$ and $\mathcal{C}_2(x)$ and derived that $\mathcal{D}_2(x)=O(x^{3/2+\varepsilon})$. These differences were also studied by Furuya \cite{MR2176479}, he gave the asymptotic formula
\begin{equation*}
  \mathcal{D}_2(x)=\mathcal{C}_2(x)+\frac{1}{6}x\log^2x+\frac{8\gamma-1}{12}x\log x+\frac{8\gamma^2-2\gamma+1}{12}x+\left\{
  \begin{array}{lr}
    O \\
    \Omega_\pm
  \end{array}
  \right\}(x^{3/4}\log x).
\end{equation*}

In this paper, as an analogue of the discrete mean values, we consider a sum of $\Delta^2(\cdot)$ over a subset of $\mathbb{N}$. More precisely over primes, namely $\sum_{p\leqslant x}\Delta^2(p)$, where $p$ runs over primes and $x$ is a large real variable. And the result is stated as follows:
\begin{theorem*}
Let $p$ be prime and $x\ge2$ is a large real number, then 
\begin{equation}
  \sum_{p\leqslant x}\Delta^2(p)=\frac{\mathcal{C}}{4\pi^2}\sum_{p\leqslant x}p^{1/2}+O(x^{23/16+\varepsilon})
\end{equation}
holds for any $\varepsilon>0$, where $\mathcal{C}$ is defined in (\ref{mean square}).
\end{theorem*}

\section{Priliminary}

\subsection{Notations}
Throught this paper, $\varepsilon$ denotes a sufficiently small positive number, not necessarily the same at each occurrence. Let $p$ always denote a prime number. As usual, $d(n)$ $\Lambda(n)$ and $\mu(n)$ denote the Dirichlet divisor function, the von Mangoldt function and the M\"{o}bius function, respectively, $\mathbb{C}$ denotes the set of complex numbers, $\mathbb{R}$ denotes the set of real numbers, and $\mathbb{N}$ denotes the set of natural numbers. We write $\mathbf{e}(t):=exp(2\pi it)$ . The notation  $m\sim M$ means $M<m\leqslant2M$, $f(x)\ll g(x)$ means $f(x)=O(g(x))$; $f(x)\asymp g(x)$ means $f(x)\ll g(x)\ll f(x)$. $\gcd(m_1,\cdots,m_j)$ denotes the greatest common divisor of $m_1,\cdots,m_j$ $\in\mathbb{N}$($j=2,3,\cdots$).

\subsection{Auxiliary Lemmas}
\
\newline \indent We need the following lemmas.

\begin{lemma}\label{voronoi}
For real numbers $N$ and $x$ satisfying $1\ll N\ll x$, we have
 \begin{equation*}
 \Delta(x)=\delta_1(x,N)+\delta_2(x,N),
 \end{equation*} 
where
\begin{equation*}
  \delta_1(x,N)=\frac{x^{1/4}}{\sqrt{2}\pi}\sum_{n\leqslant N}\frac{d(n)}{n^{3/4}}\cos\left(4\pi\sqrt{nx}-\frac{\pi}{4}\right),
\end{equation*}
and $\delta_2(x,N)$ satisfies:

 (i) $\delta_2(x,N)\ll x^{1/2}N^{-1/2}$.

(ii) For any fixed real number $T\geqslant2$, we have
\begin{equation*}
  \int_{T}^{2T}\delta_2(x,N)dx\ll\frac{T^{3/2}\log^3T}{N^{1/2}}+T\log^4T.
\end{equation*}
\end{lemma}
\begin{proof}
See, for example, the following references: Lau and Tsang\cite{MR2475967}, Tsang\cite{MR1162488}, Zhai\cite{MR2067871}.
\end{proof}

\begin{lemma}\label{H-B identity}
  Let $z\geqslant1$ be a real number and $k\geq1$ be an integer. Then for any $n\leqslant 2z^k$, there holds
\begin{equation}
  \Lambda(n)=\sum_{j=1}^{k}(-1)^{j-1}\binom{k}{j}\mathop{\sum\cdots\sum}_{\substack{n_1n_2\cdots n_{2j}=n\\n_{j+1},\cdots,n_{2j}\leqslant z}}(\log n_1)\mu(n_{j+1})\cdots\mu(n_{2j}).
\end{equation}
\end{lemma}
\begin{proof}
 See the argument on pp.1366-1367 of Heath-Brown\cite{MR0678676}.
\end{proof}

\begin{lemma}\label{Double large sieve}
  Let $\mathscr{X}:=\{x_r\}$, $\mathscr{Y}:=\{y_s\}$ be two finite sequence of real numbers with 
\begin{equation*}
  |x_r|\leqslant X, \qquad|y_s|\leqslant Y,
\end{equation*}
and $\varphi_r,\psi_s\in\mathbb{C}$. Defining the bilinear forms
\begin{equation*}
  {\mathscr{B}}_{\varphi\psi}(\mathscr{X},\mathscr{Y})=\sum_{r\sim R}\sum_{s\sim S}\varphi_r\psi_s\mathbf{e}(x_ry_s),
\end{equation*}
then we have
\begin{equation*}
  |\mathscr{B}_{\varphi\psi}(\mathscr{X},\mathscr{Y})|^2\leqslant20(1+XY)\mathscr{B}_{\varphi}(\mathscr{X},Y)\mathscr{B}_{\psi}(\mathscr{Y},X),
\end{equation*}\
with
\begin{equation*}
  \mathscr{B}_{\varphi}(\mathscr{X},Y)=\sum_{|x_{r_1}-x_{r_2}|\leqslant Y^{-1}}|\varphi_{r_1}\varphi_{r_2}|
\end{equation*}
and $\mathscr{B}_{\psi}(\mathscr{Y},X)$ defined similarly.
\end{lemma}
\begin{proof}
  See Proposition 1 of Fouvry and Iwaniec\cite{MR1027058}.
\end{proof}

\begin{lemma}\label{(2.2) of Zhai}
  Let $\{a_i\}$ and $\{b_i\}$ be two finite sequences of real numbers and let $\delta>0$ be a real number. Then we have
\begin{equation*}
  \#\{(r,s):|a_r-b_s|\leqslant\delta\}\leqslant3(\#\{(r,r'):|a_r-a_{r'}|\leqslant\delta\})^{1/2}(\#\{(s,s'):|b_s-b_{s'}|\leqslant\delta\})^{1/2}
\end{equation*}
\end{lemma}
\begin{proof}
  See the argument of P266-267 in Zhai\cite{MR2168766}.
\end{proof}

\begin{lemma}\label{Robert and Sargos}
  Let $\alpha\neq0,1$ be a fixed real integer. For any integer $M\geqslant2$ and any real number $\delta>0$, we define $\mathcal{N}(M,\delta)$ as being the number of quadruplets
\begin{equation*}
  \underline{m}=(m_1,m_2,m_3,m_4)\in\{M+1,M+2,\cdots,2M\}^4
\end{equation*}
which satisfy the inequality
\begin{equation*}
  |{m_1}^{\alpha}+{m_2}^{\alpha}-{m_3}^{\alpha}-{m_4}^{\alpha}|\leqslant\delta M^{\alpha},
\end{equation*}
then we have
\begin{equation*}
  \mathcal{N}(M,\delta)\ll M^{2+\varepsilon}+\delta M^{4+\varepsilon},
\end{equation*}
where the implied constant depends only on $\varepsilon$.
\end{lemma}
\begin{proof}
  See the argument on Theorem 2 of Robert and Sargos\cite{MR2212877}.
\end{proof}

\begin{lemma}\label{iwaniec and sarkozy}
  Let $N$ be a given large integer, $X$ be a real number and $\mathbf{N}$ denote the set $\{N+1,\cdots,2N\}$. Suppose 
\begin{equation*}
  v(\mathbf{N},X)=|\{n_1,n_2\in\mathbf{N};|\sqrt{n_1}-\sqrt{n_2}|\leqslant(2X)^{-1}\}|,
\end{equation*}
then we have
\begin{equation*}
  v(\mathbf{N},X)\leqslant(1+2\sqrt{2N}X^{-1})|\mathbf{N}|.
\end{equation*}
\end{lemma}
\begin{proof}
  See Lemma 2 of Iwaniec and S\'{a}rk\"{o}zy\cite{MR0883536}.
\end{proof}

\begin{lemma}\label{sum to int}
  Let $a,b$ are fixed real numbers and $a<b$. Suppose $f(x)$ is a real function with $|f'(x)|\leqslant1-\theta$ and $f''(x)\neq0$ on $[a,b]$. We then have
\begin{equation*}
  \sum_{a<n\leqslant b}\mathbf{e}(f(n))=\int_{a}^{b}\mathbf{e}(f(x))dx+O(\theta^{-1}).
\end{equation*}
\end{lemma}
\begin{proof}
  See Lemma 8.8 of Iwaniec and Kowalski\cite{MR2061214}.
\end{proof}

\begin{lemma}\label{first derivative}
  Let $F(x)$ be a real differentiable function such that $F'(x)$ is monotonic and $|F'(x)|\geqslant m>0$ for $a\leqslant x\leqslant b$. Then
\begin{equation*}
  \left|\int_{a}^{b}e^{iF(x)}dx\right|\leqslant4m^{-1}.
\end{equation*}
\end{lemma}
\begin{proof}
  See Lemma 2.1 of Ivi\'{c}\cite{MR0792089}.
\end{proof}

\begin{lemma}\label{deravative est}
  Suppose that $f(x):[a,b]\rightarrow\mathbb{R}$ has continuous derivatives of arbitrary order on $[a,b]$, where $1\leqslant a<b\leqslant2a$. Suppose further that we have 
\begin{equation*}
  |f''(x)|\asymp\lambda_2,\qquad x\in[a,b].
\end{equation*}
Then
\begin{equation}\label{second order derivative}
  \sum_{a<n\leqslant b}\mathbf{e}(f(n))\ll a{\lambda_2}^{1/2}+{\lambda_2}^{-1/2}.
\end{equation}
\end{lemma}
\begin{proof}
 See Corollary 8.13 of Iwaniec and Kowalski\cite{MR2061214}, or Theorem 5 of Chapter 1 in Karatsuba\cite{MR1215269}.
\end{proof}

\begin{lemma}\label{Heath and Tsang}
Let $T$ be a large real number and $H$ be a real number such that $1\leqslant H\leqslant T$. We have
\begin{equation*}
\begin{aligned}
  \int_{T}^{2T}\max_{h\leqslant H}(\Delta(t+h)-\Delta(t))^2dt\ll HT\log^5T.
\end{aligned}
\end{equation*}
\end{lemma}
\begin{proof}
  See Heath-Brown and Tsang\cite{MR1295953}.
\end{proof}

\begin{lemma}\label{T(a,b)}
  Suppose $N_1,N_2\geqslant1$ are given real numbers, $\alpha,\beta$ are real numbers such that $0<\alpha,\beta\leqslant1$. Define
\begin{equation*}
  T(N_1,N_2,\alpha,\beta)=\sum_{\substack{n_1\sim N_1,n_2\sim N_2\\n_1\neq n_2}}\frac{1}{|{n_1}^{\alpha}-{n_2}^{\alpha}|^{\beta}}.
\end{equation*}  
Then we have
\begin{equation*}
  T(N_1,N_2,\alpha,\beta)\ll(N_1N_2)^{1-\alpha\beta/2}\log N_1\log N_2.
\end{equation*}
\end{lemma}
\begin{proof}
  We divide the sum into
\begin{equation*}
\begin{aligned}
  \sum_{\substack{n_1\sim N_1\\n_2\sim N_2}}(|{n_1}^{\alpha}-{n_2}^{\alpha}|)^{-\beta}={\sum}_1+{\sum}_2,
\end{aligned} 
\end{equation*}
where
\begin{equation*}
\begin{aligned}
  {\sum}_1&=\sum_{\substack{n_1\sim N_1,n_2\sim N_2\\|{n_1}^{\alpha}-{n_2}^{\alpha}|\leqslant\frac{(n_1n_2)^{\alpha/2}}{100}}}
  (|{n_1}^{\alpha}-{n_2}^{\alpha}|)^{-\beta},\\
  {\sum}_2&=\sum_{\substack{n_1\sim N_1,n_2\sim N_2\\|{n_1}^{\alpha}-{n_2}^{\alpha}|\geqslant\frac{(n_1n_2)^{\alpha/2}}{100}}}
  (|{n_1}^{\alpha}-{n_2}^{\alpha}|)^{-\beta}.
\end{aligned} 
\end{equation*}
Since $|{n_1}^{\alpha}-{n_2}^{\alpha}|\leqslant\frac{(n_1n_2)^{\alpha/2}}{100}$, we have $n_1\asymp n_2$ in ${\sum}_1$. By the Lagrange mean value Theorem we have
\begin{equation*}
\begin{aligned}
  {\sum}_1&\ll\sum_{\substack{n_1\sim N_1,n_2\sim N_2\\n_1\asymp n_2,n_1\neq n_2}}\frac{1}{{n_1}^{(\alpha-1)\beta}|n_1-n_2|^{\beta}}\\
  &\ll\sum_{n_1\sim N_1}\frac{1}{{n_1}^{(\alpha-1)\beta}}\sum_{\substack{n_2\sim N_2\\n_1\asymp n_2\\n_1\neq n_2}}\frac{1}{|n_1-n_2|^{\beta}},
\end{aligned} 
\end{equation*}
let $|n_1-n_2|=t$ in the latter sum, we have
\begin{equation*}
\begin{aligned}
  {\sum}_1&\ll\sum_{n_1\sim N_1}\frac{1}{{n_1}^{(\alpha-1)\beta}}\sum_{t\leqslant n_1}t^{-\beta}\\
  &\ll\sum_{n_1\sim N_1}\frac{1}{{n_1}^{\alpha\beta-1}}\log {N_1}\\
  &\ll{N_1}^{2-\alpha\beta}\log{N_1}.
\end{aligned} 
\end{equation*}
Similarly we have ${\sum}_1\ll{N_2}^{2-\alpha\beta}\log{N_2}$, thus 
\begin{equation}\label{sum1}
  {\sum}_1\ll(N_1N_2)^{1-\alpha\beta/2}\log{N_1}\log{N_2}.
\end{equation}
For ${\sum}_2$,
\begin{equation}\label{sum2}
  {\sum}_2\ll\sum_{n_1\sim N_1,n_2\sim N_2}\frac{1}{(n_1n_2)^{\alpha\beta/2}}\ll(N_1N_2)^{1-\alpha\beta/2}.
\end{equation}
Thus we finish the proof by (\ref{sum1}) and (\ref{sum2}).
\end{proof}

We state a multiple exponential sums of the following form:
\begin{equation}\label{exp sum wlog}
  \mathop{\sum_{m_1\sim M_1}\sum_{m_2\sim M_2}}_{\substack{m_1\neq m_2}}d(m_1)d(m_2)\mathop{\sum_{h\sim H}\sum_{\ell\sim L}}_{\substack{HL\asymp x}}\xi_h\eta_{\ell}\mathbf{e}\left(2(\sqrt{m_1}-\sqrt{m_2})h^{1/2}{\ell}^{1/2}\right),
\end{equation}
with $\xi_h, \eta_\ell\in\mathbb{C}, |\xi_h|\ll x^{\varepsilon}, |\eta_\ell|\ll x^{\varepsilon}$ and $M_1,M_2,H,L,x$ are given real numbers, such that $1\ll M_1,M_2\ll x$. It is called a "Type I" sum, denoted by $S_I$, if $\eta_\ell=1$ or $\eta_\ell=\log\ell$; otherwise it is called a "Type II" sum, denoted by $S_{II}$.

\begin{lemma}\label{Type I est}
  Suppose that $\xi_h\ll1$, $\eta_\ell=1$ or $\eta_\ell=\log\ell$, $HL\asymp x$. Then if $H\ll x^{1/4}$, there holds
\begin{equation*}
  x^{-\varepsilon}S_I\ll x^{1/2}{M_1}^{5/4}M_2+x^{3/4}(M_1M_2)^{7/8}.
\end{equation*}
\end{lemma}
\begin{proof}
  Set $f(\ell)=2({m_1}^{1/2}-{m_2}^{1/2})(h\ell)^{1/2}$. It is easy to see that 
\begin{equation*}
  f^{''}(\ell)=-\frac{1}{2}\ell^{-3/2}h^{1/2}(\sqrt{m_1}-\sqrt{m_2})\asymp|\sqrt{m_1}-\sqrt{m_2}|L^{-3/2}H^{1/2}.
\end{equation*}
If $H\ll x^{1/4}$, then by Lemma \ref{deravative est}, we deduce that
\begin{equation*}
\begin{aligned}
  x^{-\varepsilon}S_I&\ll\mathop{\sum_{m_1\sim M_1}\sum_{m_2\sim M_2}}_{\substack{m_1\neq m_2}}\sum_{h\sim H}\left|\sum_{\ell\sim L}\mathbf{e}(f(\ell))\right|\\
  &\ll\mathop{\sum_{m_1\sim M_1}\sum_{m_2\sim M_2}}_{\substack{m_1\neq m_2}}\sum_{h\sim H}\left(L(L^{-3/2}H^{1/2}|\sqrt{m_1}-\sqrt{m_2}|)^{1/2}+(L^{-3/2}H^{1/2}|\sqrt{m_1}-\sqrt{m_2}|)^{-1/2}\right)\\
  &\ll x^{1/4}H{M_1}^{5/4}M_2+x^{3/4}\mathop{\sum_{m_1\sim M_1}\sum_{m_2\sim M_2}}_{\substack{m_1\neq m_2}}|\sqrt{m_1}-\sqrt{m_2}|^{-1/2},
\end{aligned}
\end{equation*}
taking $n_1=m_1,n_2=m_2$, $\alpha=\beta=1/2$ in Lemma \ref{T(a,b)} we have
\begin{equation*}
\begin{aligned}
  x^{-\varepsilon}S_I&\ll x^{1/4}H{M_1}^{5/4}M_2+x^{3/4}(M_1M_2)^{7/8}\\
  &\ll x^{1/2}{M_1}^{5/4}M_2+x^{3/4}(M_1M_2)^{7/8}.
\end{aligned}
\end{equation*}
\end{proof}

In order to separate the dependence from the range of summation we appeal to the following formula
\begin{lemma}\label{separate}
  Let $0<V\leqslant U<\nu U<\lambda V$ and let $a_{v}$ be complex numbers with $|a_{v}|\leqslant1$. We then have
\begin{equation*}
  \sum_{U<u<\nu U}a_u=\frac{1}{2\pi}\int_{-V}^{V}\left(\sum_{V<v<\lambda V}a_vv^{-it}\right)U^{it}(\nu^{it}-1)t^{-1}dt+O(\log(2+V)),
\end{equation*}
where the constant implied in $O$ depends on $\lambda$ only.
\end{lemma}
\begin{proof}
  See Lemma 6 of Fouvry and Iwaniec\cite{MR1027058}.
\end{proof}

\begin{lemma}\label{Type II est}
Suppose $|\xi_h|\ll x^{\varepsilon}, |\eta_\ell|\ll x^{\varepsilon}$ with $h\sim H$, $\ell\sim L$, $HL\asymp x$. Then for $x^{1/4}\ll H\ll{x^{1/2}}$, there holds
\begin{equation*}
\begin{aligned}
  x^{-\varepsilon}S_{II}&\ll x^{3/4}{M_1}^{9/8}{M_2}^{7/8}+x^{7/8}M_1{M_2}^{3/4}+x^{13/16}{M_1}^{19/16}{M_2}^{3/4}+x^{15/16}(M_1M_2)^{3/4}\\
  &+x^{1/2}M_1M_2.
\end{aligned}
\end{equation*}
\end{lemma}
\begin{proof}
Applying Lemma \ref{separate} we obtain 
\begin{equation}\label{exp SII}
  S_{II}\ll|{S_{II}}^{*}|+M_1M_2H\log x,
\end{equation}
where
\begin{equation*}
  {S_{II}}^{*}=\mathop{\sum_{m_1\sim M_1}\sum_{m_2\sim M_2}}_{\substack{m_1\neq m_2}}d(m_1)d(m_2)\sum_{h\sim H}{\xi_h}^{*}\sum_{\ell\sim L}{\eta_{\ell}}^{*}\mathbf{e}\left(2(\sqrt{m_1}-\sqrt{m_2})h^{1/2}{\ell}^{1/2}\right)
\end{equation*}
with $|{\xi_h}^{*}|,|{\eta_{\ell}}^{*}|\ll x^{\varepsilon}$.

Applying Lemma \ref{Double large sieve} with $\mathscr{X}=2({m_1}^{1/2}-{m_2}^{1/2})h^{1/2}$, $\mathscr{Y}=\ell^{1/2}$, we obtain
\begin{equation}\label{T * est1}
  {S_{II}}^{*}\ll{M_1}^{1/4}x^{1/4}(S_{a}S_{b})^{1/2},
\end{equation}
where
\begin{equation*}
\begin{aligned}
  S_{a}=&\sum_{|({m_1}^{1/2}-{m_2}^{1/2}){h}^{1/2}-({\tilde{m_1}}^{1/2}-{\tilde{m_2}}^{1/2}){\tilde{h}}^{1/2}|\leqslant L^{{-1/2}}}d(m_1)d(m_2)d(\tilde{m_1})d(\tilde{m_2}){\xi_h}^{*}{\xi_{\tilde{h}}}^{*},\\
  S_{b}=&\sum_{|\ell^{1/2}-{\tilde{\ell}}^{1/2}|\leqslant{M_1}^{-1/2}H^{-1/2}}{\eta_{\ell}}^{*}{\eta_{\tilde{\ell}}}^{*}.
\end{aligned}
\end{equation*}
By (\ref{mathscr B est}) with $\Delta=L^{-1/2}$, we easily obtain
\begin{equation}\label{S_a est}
\begin{aligned}
  x^{-\varepsilon}S_{a}&\ll\mathscr{B}(M_1,M_2,H,L^{-1/2})\\
  &\ll(M_1M_2)^{3/2}H^{3/2}+{M_1}^{15/8}{M_2}^{3/2}x^{-1/8}H^{3/2}+(M_1M_2)^{7/4}x^{-1/2}H^2.
\end{aligned}
\end{equation}
For $S_b$, using Lemma \ref{iwaniec and sarkozy} we have
\begin{equation}\label{S_b est}
\begin{aligned}
  x^{-\varepsilon}S_{b}&\ll \sum_{|\ell^{1/2}-{\tilde{\ell}}^{1/2}|\leqslant{M_1}^{-1/2}H^{-1/2}}1\\
  &\ll xH^{-1}+{M_1}^{-1/2}H^{-2}x^{3/2}.
\end{aligned}
\end{equation}
Thus combining (\ref{exp SII}), (\ref{T * est1}), (\ref{S_a est}), (\ref{S_b est}) and the condition $x^{1/4}\ll H\ll x^{1/2}$ we conclude that
\begin{equation}\label{T * est2}
\begin{aligned}
  x^{-\varepsilon}{S_{II}}^{*}&\ll x^{1/2}H^{1/2}{M_1}^{9/8}{M_2}^{7/8}+x^{3/4}H^{1/4}M_1{M_2}^{3/4}+H^{1/4}x^{11/16}{M_1}^{19/16}{M_2}^{3/4}\\
  &+x^{3/4}(M_1M_2)^{7/8}+xH^{-1/4}(M_1M_2)^{3/4}+x^{15/16}H^{-1/4}{M_1}^{15/16}{M_2}^{3/4}+x^{1/2}M_1M_2\\
  &\ll x^{3/4}{M_1}^{9/8}{M_2}^{7/8}+x^{7/8}M_1{M_2}^{3/4}+x^{13/16}{M_1}^{19/16}{M_2}^{3/4}+x^{15/16}(M_1M_2)^{3/4}\\
  &+x^{1/2}M_1M_2.
\end{aligned} 
\end{equation}
\end{proof}

\section{The spacing problem}

Let $M_1,M_2\geqslant1$, $H\geqslant1$. In this section $\Delta$ is a small real positive number. In this section we investigate the distributions of real numbers of type
\begin{equation*}
  t(h,m_1,m_2)=({m_1}^{\alpha}-{m_2}^{\alpha})h^{\beta}
\end{equation*}
with $0<\alpha,\beta<1$, $m_1\sim M_1$, $m_2\sim M_2$, $m_1\neq m_2$, $h\sim H$.

Let $\mathscr{B}(M_1, M_2, H, \Delta)$ denote the number of sixturplets $(m_1, \tilde{m_1}, m_2, \tilde{m_2}, h, \tilde{h})$ with $m_1, \tilde{m_1}\sim M_1$, $m_2, \tilde{m_2}\sim M_2$, $h, \tilde{h}\sim H$ satisfying
\begin{equation}\label{mathscr B}
  |t(h,m_1,m_2)-t(\tilde{h},\tilde{m_1},\tilde{m_2})|\leqslant\Delta
\end{equation}

Our aim is to prove the following:

\begin{proposition}
  We have
\begin{equation}\label{mathscr B est}
\begin{aligned}
  &\mathscr{B}(M_1, M_2, H, \Delta)\\
  &\ll(M_1M_2)^{3/2+\varepsilon}H^{3/2}+{M_1}^{2-\alpha/4+\varepsilon}{M_2}^{3/2+\varepsilon}\Delta^{1/4}H^{3/2-\beta/4}+(M_1M_2)^{2-\alpha/2+\varepsilon}H^{2-\beta}\Delta.
\end{aligned}
\end{equation}
\end{proposition}
\begin{proof}
Suppose $m_1,m_2,\tilde{m_1},\tilde{m_2}$ are fixed, without loss of generality we assume that $m_1>m_2$ and $\tilde{m_1}>\tilde{m_2}$. We denote the number of solutions of
\begin{equation}\label{mathscr B1}
  \left|\left(\frac{{m_1}^{\alpha}-{m_2}^{\alpha}}{{\tilde{m_1}}^{\alpha}-{\tilde{m_2}}^{\alpha}}\right)
  -\left(\frac{h}{\tilde{h}}\right)^{\beta}\right|\leq\frac{\Delta}{|{\tilde{m_1}}^{\alpha}-{\tilde{m_2}}^{\alpha}|H^{\beta}}
\end{equation}
by $\mathscr{B}(m_1, m_2,\tilde{m_1}, \tilde{m_2}, H, \Delta)$. Let $\mathscr{B}(m_1, m_2,\tilde{m_1}, \tilde{m_2}, H, \Delta,\mu)$ be the number of solutions of $h,\tilde{h}$ to (\ref{mathscr B1}) which additionally satisfying $\gcd(h,\tilde{h})=\mu$.

We divide the solutions of $h$, $\tilde{h}$ into classes each one having fixed values $\gcd(h,\tilde{h})=\mu$, the points $\left(\frac{h}{\tilde{h}}\right)$ are spaced by $c(\beta)H^2\mu^{-2}$, where $c(\beta)$ is a constant. Then 
\begin{equation*}
  \mathscr{B}(m_1, m_2,\tilde{m_1}, \tilde{m_2}, H, \Delta,\mu)\ll1+\Delta(|{\tilde{m_1}}^{\alpha}-{\tilde{m_2}}^{\alpha}|H^{\beta})^{-1}H^2\mu^{-2}
\end{equation*}

 Let $\mathscr{B}(M_1, M_2, H, \Delta,\mu)$ be the number of solutions to (\ref{mathscr B}) which additionally satisfying $\gcd(h,\tilde{h})=\mu$. Summing over $m_1$, $m_2$, $\tilde{m_1}$, $\tilde{m_2}$ we have
\begin{equation*}
\begin{aligned}
  \mathscr{B}(M_1, M_2, H, \Delta,\mu)&\ll \sum_{\substack{m_1,\tilde{m_1}\sim M_1\\m_2,\tilde{m_2}\sim M_2}}\mathscr{B}(m_1, m_2,\tilde{m_1}, \tilde{m_2}, H, \Delta,\mu)\\
  &\ll\sum_{\substack{m_1,\tilde{m_1}\sim M_1\\m_2,\tilde{m_2}\sim M_2}}(1+\Delta(|{\tilde{m_1}}^{\alpha}-{\tilde{m_2}}^{\alpha}|)^{-1}H^{2-\beta}\mu^{-2})\\
  &\ll{M_1}^2{M_2}^2+\Delta M_1M_2 H^{2-\beta}\mu^{-2}\sum_{\substack{\tilde{m_1}\sim M_1\\\tilde{m_2}\sim M_2}}(|{\tilde{m_1}}^{\alpha}-{\tilde{m_2}}^{\alpha}|)^{-1}.
\end{aligned} 
\end{equation*}
Using Lemma \ref{T(a,b)} we conclude that
\begin{equation*}
\begin{aligned}
  \mathscr{B}(M_1, M_2, H, \Delta,\mu)\ll{M_1}^2{M_2}^2+\Delta(M_1M_2)^{2-\alpha/2}H^{2-\beta}\mu^{-2}\log{M_1}\log{M_2}.
\end{aligned} 
\end{equation*}
Summing over $\mu$ we have
\begin{equation}\label{h-aspect}
\begin{aligned}
  \mathscr{B}(M_1, M_2, H, \Delta)\ll H{M_1}^2{M_2}^2+\Delta(M_1M_2)^{2-\alpha/2}H^{2-\beta}\log{M_1}\log{M_2}.
\end{aligned} 
\end{equation}
On the other hand, for fixed $h$, $\tilde{h}$, we denote the number of solutions of $m_1,\tilde{m_1},m_2,\tilde{m_2}$ to
\begin{equation}\label{mathscr B2}
  \left|\left(\frac{h}{\tilde{h}}\right)^{\beta}({m_1}^{\alpha}-{m_2}^{\alpha})-{\tilde{m_1}}^{\alpha}+{\tilde{m_2}}^{\alpha}\right|\leq\frac{\Delta}{H^{\beta}}
\end{equation}
by $\mathscr{B}(M_1, M_2, h, \tilde{h}, H, \Delta)$. Then we can obtain that $\mathscr{B}(M_1, M_2, H, \Delta)$ is bounded by
\begin{equation*}
   \sum_{h,\tilde{h}\sim H}\mathscr{B}(M_1, M_2, h, \tilde{h}, H, \Delta).
\end{equation*}

Applying Lemma \ref{(2.2) of Zhai} to the sequences $\{a_r\}=\{(h/\tilde{h})^{\beta}{m_1}^{\alpha}-{\tilde{m_1}}^{\alpha}\}$ and $\{b_s\}=\{(h/\tilde{h})^{\beta}{m_2}^{\alpha}-{\tilde{m_2}}^{\alpha}\}$, we get
\begin{equation*}
\begin{aligned}
  &\mathscr{B}(M_1, M_2, h, \tilde{h}, H, \Delta)\\
  &\qquad\leqslant3\left(\#\left\{(m_1,\tilde{m_1},{m_1}',{\tilde{m_1}'}):
  \left|\left(\frac{h}{\tilde{h}}\right)^{\beta}{m_1}^{\alpha}-{\tilde{m_1}}^{\alpha}-\left(\frac{h}{\tilde{h}}\right)^{\beta}{{m_1}'}^{\alpha}+{\tilde{m_1}{'}}^{\alpha}\right|
  \leqslant\frac{\Delta}{H^{\beta}}\right\}\right)^{1/2}\\
  &\qquad\times\left(\#\left\{(m_2,\tilde{m_2},{m_2}',{\tilde{m_2}'}):
  \left|\left(\frac{h}{\tilde{h}}\right)^{\beta}{m_2}^{\alpha}-{\tilde{m_2}}^{\alpha}-\left(\frac{h}{\tilde{h}}\right)^{\beta}{{m_2}'}^{\alpha}+{\tilde{m_2}{'}}^{\alpha}\right|
  \leqslant\frac{\Delta}{H^{\beta}}\right\}\right)^{1/2}.
\end{aligned} 
\end{equation*}
Applying Lemma \ref{(2.2) of Zhai} again to the sequences $\{a_{r_1}\}=\{(h/\tilde{h})^{\beta}({m_1}^{\alpha}-{{m_1}{'}^{\alpha}})\}$,
$\{b_{s_1}\}=\{{\tilde{m_1}}^{\alpha}-{\tilde{m_1}{'}}^{\alpha}\}$, and $\{a_{r_2}\}=\{(h/\tilde{h})^{\beta}({m_2}^{\alpha}-{{m_2}{'}^{\alpha}})\}$, $\{b_{s_2}\}=\{{\tilde{m_2}}^{\alpha}-{\tilde{m_2}{'}}^{\alpha}\}$, respectively, we get 
\begin{equation*}
\begin{aligned}
  \mathscr{B}(M_1, M_2, h, \tilde{h}, H, \Delta)
  \ll\mathcal{N}\left(M_1,\frac{\Delta}{H^{\beta}{M_1}^{\alpha}}\right)^{1/2}\mathcal{N}\left(M_2,\frac{\Delta}{H^{\beta}{M_2}^{\alpha}}\right)^{1/2},
\end{aligned}
\end{equation*}
where $\mathcal{N}$ is defined in Lemma \ref{Robert and Sargos}. Summing over $h$ and $\tilde{h}$ we have
\begin{equation}\label{m-aspect}
\begin{aligned}
  \mathscr{B}(M_1, M_2, H, \Delta)&\ll\sum_{h,\tilde{h}\sim H}\mathscr{B}(M_1, M_2, h, \tilde{h}, H, \Delta)\\
  &\ll(H{M_1}^{1+\varepsilon}+{M_1}^{2-\alpha/2+\varepsilon}\Delta^{1/2}H^{1-\beta/2})(H{M_2}^{1+\varepsilon}+{M_2}^{2-\alpha/2+\varepsilon}\Delta^{1/2}H^{1-\beta/2})
\end{aligned}
\end{equation}
Above all, combining (\ref{h-aspect}) and (\ref{m-aspect}) we get
\begin{equation*}
\begin{aligned}
  &\mathscr{B}(M_1, M_2, H, \Delta)\\
  \ll&\min({M_1}^2{M_2}^2+\Delta(M_1M_2)^{2-\alpha/2}H^{2-\beta}\log{M_1}\log{M_2},\\
  &(H{M_1}^{1+\varepsilon}+{M_1}^{2-\alpha/2+\varepsilon}\Delta^{1/2}H^{1-\beta/2})(H{M_2}^{1+\varepsilon}+{M_2}^{2-\alpha/2+\varepsilon}\Delta^{1/2}H^{1-\beta/2}))\\
  \ll&\min(H{M_1}^2{M_2}^2,H^2(M_1M_2)^{1+\varepsilon}+{M_1}^{2-\alpha/2+\varepsilon}{M_2}^{1+\varepsilon}\Delta^{1/2}H^{2-\beta/2})+\Delta(M_1M_2)^{2-\alpha/2+\varepsilon}H^{2-\beta}\\
  \ll&(M_1M_2)^{3/2+\varepsilon}H^{3/2}+{M_1}^{2-\alpha/4+\varepsilon}{M_2}^{3/2+\varepsilon}\Delta^{1/4}H^{3/2-\beta/4}+(M_1M_2)^{2-\alpha/2+\varepsilon}H^{2-\beta}\Delta,
\end{aligned}
\end{equation*}
where we use $\min(a,b)\ll (ab)^{1/2}$.
\end{proof}

\section{Proof of the Theorem}

We are first going to estimate $\sum_{x<p\leqslant2x}\Delta^2(p)$. Using notations in Lemma \ref{voronoi}, let $M$ be a real number such that $1\ll M\ll x$, one has
\begin{equation}\label{deltap2 exp}
\begin{aligned}
  \sum_{x<p\leqslant2x}\Delta^2(p)
  =\sum_{x<p\leqslant2x}(\delta_1(p,M)^2+2\delta_1(p,M)\delta_2(p,M)+\delta_2(p,M)^2),
\end{aligned}
\end{equation}
denoted by $S_{11},S_{12},S_{22}$, respectively.

For $S_{22}$, one has
\begin{equation*}
\begin{aligned}
  S_{22}&=\sum_{x<p\leqslant2x}\int_{p-1}^{p}\delta_2(p,M)^2dt\\
  &=\sum_{x<p\leqslant2x}\left[\int_{p-1}^{p}\delta_2(t,M)^2dt-\int_{p-1}^{p}(\delta_2(t,M)^2-\delta_2(p,M)^2)dt\right]\\
  &\leqslant\int_{x}^{2x}\delta_2(t,M)^2dt+\sum_{x<p\leqslant2x}\int_{p-1}^{p}(\delta_2(t,M)^2-\delta_2(p,M)^2)dt\\
  &\ll\frac{x^{3/2}\log^3x}{{M}^{1/2}}+x\log^4x+\sum_{x<p\leqslant2x}\int_{p-1}^{p}(\delta_2(t,M)^2-\delta_2(p,M)^2)dt.
\end{aligned}
\end{equation*}
The latter term can be transformed into $E_1-E_2$, where
\begin{equation*}
\begin{aligned}
  E_1&=\sum_{x<p\leqslant2x}\int_{p-1}^{p}(\delta_2(t,M)+\delta_2(p,M))(\Delta(t)-\Delta(p))dt,\\
  E_2&=\sum_{x<p\leqslant2x}\int_{p-1}^{p}(\delta_2(t,M)+\delta_2(p,M))(\delta_1(t,M)-\delta_1(p,M))dt.
\end{aligned}
\end{equation*}
For $E_2$, by the expression of $\delta_1$ and Lagrange's mean value theorem we obtain $\delta_1(t,M)-\delta_1(p,M)\ll t^{-1/4}{M}^{3/4}$, and by Lemma \ref{voronoi} we have $\delta_2(t,M)+\delta_2(p,M)\ll t^{1/2+\varepsilon}{M}^{-1/2}$, thus 
\begin{equation*}
  E_2\ll\int_{x}^{2x}t^{1/4+\varepsilon}{M}^{1/4}dt\ll x^{5/4+\varepsilon}{M}^{1/4}.
\end{equation*}
For $E_1$, one has
\begin{equation*}
\begin{aligned}
  E_1&\ll\int_{x}^{2x}t^{1/2}{M}^{-1/2}\max_{0<v\leqslant1}|\Delta(t)-\Delta(t+v)|dt\\
  &\ll \frac{x^{1/2+\varepsilon}}{{M}^{1/2}}\int_{x}^{2x}\max_{0<v\leqslant1}|\Delta(t)-\Delta(t+v)|dt\\
  &\ll x^{3/2+\varepsilon}{M}^{-1/2},
\end{aligned}
\end{equation*}
where we use Lemma \ref{Heath and Tsang}. Thus we conclude that 
\begin{equation}\label{S_22 est}
  S_{22}\ll x^{3/2+\varepsilon}M^{-1/2}+x^{5/4+\varepsilon}{M}^{1/4}.
\end{equation}

Next we turn to evaluate $S_{11}$. Using $\cos{\alpha}\cos{\beta}=\frac{1}{2}(\cos(\alpha-\beta)+\cos(\alpha+\beta))$ we have
\begin{equation}\label{S11 exp}
\begin{aligned}
  &\sum_{x<p\leqslant2x}\delta_1(p,M)^2\\
  =&\sum_{x<p\leqslant2x}\frac{p^{1/2}}{2\pi^2}\sum_{m_1,m_2\leqslant M}\frac{d(m_1)d(m_2)}{(m_1m_2)^{3/4}}
  \cos\left(4\pi\sqrt{m_1p}-\frac{\pi}{4}\right)\cos\left(4\pi\sqrt{m_2p}-\frac{\pi}{4}\right)\\
  =&\frac{1}{4\pi^2}\sum_{x<p\leqslant2x}p^{1/2}\sum_{m_1,m_2\leqslant M}\frac{d(m_1)d(m_2)}{(m_1m_2)^{3/4}}
  \left(\cos\left(4\pi(\sqrt{m_1}-\sqrt{m_2})\sqrt{p}\right)+\sin\left(4\pi(\sqrt{m_1}+\sqrt{m_2})\sqrt{p}\right)\right)\\
  :=&\frac{1}{4\pi^2}\sum_{x<p\leqslant2x}p^{1/2}\sum_{m\leqslant M}\frac{d^2(m)}{m^{3/2}}+R_1+R_2,
\end{aligned}
\end{equation}
where
\begin{equation*}
\begin{aligned}
R_1&=\frac{1}{4\pi^2}\sum_{\substack{m_1,m_2\leqslant M\\m_1\neq m_2}}\frac{d(m_1)d(m_2)}{(m_1m_2)^{3/4}}\sum_{x<p\leqslant2x}p^{1/2}\cos\left(4\pi(\sqrt{m_1}-\sqrt{m_2})\sqrt{p}\right),\\
R_2&=\frac{1}{4\pi^2}\sum_{m_1,m_2\leqslant M}\frac{d(m_1)d(m_2)}{(m_1m_2)^{3/4}}\sum_{x<p\leqslant2x}p^{1/2}\sin\left(4\pi(\sqrt{m_1}+\sqrt{m_2})\sqrt{p}\right).
\end{aligned}
\end{equation*}
\subsection{Estimate of \texorpdfstring{$R_1$}{}}
\
\newline \indent  One can see $R_1$ can be written as linear combination of $O(\log^2{M})$ sums of the form
\begin{equation*}
\begin{aligned}
  &\frac{1}{4\pi^2}\sum_{\substack{m_1\sim M_1,m_2\sim M_2\\m_1\neq m_2}}\frac{d(m_1)d(m_2)}{(m_1m_2)^{3/4}}\sum_{x<p\leqslant2x}p^{1/2}\cos\left(4\pi(\sqrt{m_1}-\sqrt{m_2})\sqrt{p}\right)\\
  =&\frac{1}{8\pi^2}\sum_{\substack{m_1\sim M_1,m_2\sim M_2\\m_1\neq m_2}}\frac{d(m_1)d(m_2)}{(m_1m_2)^{3/4}}
  \sum_{x<p\leqslant2x}p^{1/2}(\mathbf{e}(2(\sqrt{m_1}-\sqrt{m_2})\sqrt{p})+\mathbf{e}(2(\sqrt{m_2}-\sqrt{m_1})\sqrt{p})),
\end{aligned}
\end{equation*}
by using $\cos2\pi\alpha=(\mathbf{e}(\alpha)+\mathbf{e}(-\alpha))/2$, and where $1\leqslant M_1,M_2\leqslant M$. Since $m_1,m_2$ are symmetric in the above sum, we are just going to estimate
\begin{equation*}
\mathcal{S}_1=\frac{1}{8\pi^2}\sum_{\substack{m_1\sim M_1,m_2\sim M_2\\m_1\neq m_2}}\frac{d(m_1)d(m_2)}{(m_1m_2)^{3/4}}
  \sum_{x<p\leqslant2x}p^{1/2}\:\mathbf{e}(2(\sqrt{m_1}-\sqrt{m_2})\sqrt{p}).
\end{equation*}
Trivially we have
\begin{equation}\label{S_1}
  \mathcal{S}_1\ll\frac{x^{1/2}}{(M_1M_2)^{3/4}}|{\mathcal{S}_1}^{*}|,
\end{equation}
where
\begin{equation*}
  {\mathcal{S}_1}^{*}=\sum_{\substack{m_1\sim M_1,m_2\sim M_2\\m_1\neq m_2}}d(m_1)d(m_2)
  \sum_{x<p\leqslant2x}\:\mathbf{e}(2(\sqrt{m_1}-\sqrt{m_2})\sqrt{p}).
\end{equation*}
It follows from partial summation that
\begin{equation}\label{S1*}
\begin{aligned}
{\mathcal{S}_1}^{*}&=\sum_{\substack{m_1\sim M_1,m_2\sim M_2\\m_1\neq m_2}}d(m_1)d(m_2)
  \int_{x}^{2x}\frac{d\mathcal{A}(u)}{\log u}\\
  &=\sum_{\substack{m_1\sim M_1,m_2\sim M_2\\m_1\neq m_2}}d(m_1)d(m_2)
  \left(\frac{\mathcal{A}(u)}{\log u}\bigg|_{x}^{2x}-\int_{x}^{2x}\frac{\mathcal{A}(u)}{u\log^2u}du\right),
\end{aligned}
\end{equation}
where 
\begin{equation*}
  \mathcal{A}(u)=\sum_{p\leqslant u}\log p\:\mathbf{e}(2(\sqrt{m_1}-\sqrt{m_2})\sqrt{p}).
\end{equation*}
Moreover, we easily obtain
\begin{equation*}
  \mathcal{A}(u)=\sum_{n\leqslant u}\Lambda(n)\:\mathbf{e}(2(\sqrt{m_1}-\sqrt{m_2})\sqrt{n})+O(u^{1/2}).
\end{equation*}

By a splitting argument we only need to give the upper bound estimate of the following sum
\begin{equation}\label{exp sum in this paper}
\sum_{\substack{m_1\sim M_1,m_2\sim M_2\\m_1\neq m_2}}d(m_1)d(m_2)\sum_{n\sim x}\Lambda(n)\:\mathbf{e}(2(\sqrt{m_1}-\sqrt{m_2})\sqrt{n}).
\end{equation}
After using Heath-Brown's identity, i.e. Lemma \ref{H-B identity} with $k=3$, one can see that (\ref{exp sum in this paper}) can be written as linear combination of $O(\log^6x)$ sums, each of which is of the form
\begin{equation}\label{sum after H-B}
\begin{aligned}
  \mathcal{T}^{*}:=&\sum_{\substack{m_1\sim M_1,m_2\sim M_2\\m_1\neq m_2}}d(m_1)d(m_2){\sum_{n_1\sim N_1}}\cdots{\sum_{n_6\sim N_6}}(\log n_1)\mu(n_4)\mu(n_5)\mu(n_6)\\
  &\times\mathbf{e}(2(\sqrt{m_1}-\sqrt{m_2})\sqrt{n_1\cdots n_6}),
\end{aligned}
\end{equation}
where $N_1\cdots N_6\asymp x$; $2N_i\leqslant(2x)^{1/3}$, $i=4,5,6$ and some $n_i$ may only take value $1$. Therefore, it is sufficient for us to estimate for each $\mathcal{T}^{*}$ defined as in
(\ref{sum after H-B}). Next, we will consider four cases.
\subsubsection*{Case 1}
\indent If there exists an $N_j$ such that $N_j\geqslant x^{3/4}$, then we must have $j\leqslant3$ for the fact that

\qquad\;\;$N_j\ll x^{1/3}$ with $j=4,5,6$. Let
\begin{equation*}
  h=\prod_{\substack{1\leqslant i\leqslant6\\i\neq j}}n_i,\qquad \ell=n_j,\qquad H=\prod_{\substack{1\leqslant i\leqslant6\\i\neq j}}N_i,\qquad L=N_j.
\end{equation*}
In this case, we can see that $\mathcal{T}^{*}$ is a sum of "Type I" satisfying $H\ll x^{1/4}$. By Lemma \ref{Type I est}, we have
\begin{equation*}
  x^{-\varepsilon}\cdot\mathcal{T}^{*}\ll x^{1/2}{M_1}^{5/4}M_2+x^{3/4}(M_1M_2)^{7/8}.
\end{equation*}
\subsubsection*{Case 2}
\indent If there exists an $N_j$ such that $x^{1/2}\leqslant N_j< x^{3/4}$, then we take
\begin{equation*}
  h=\prod_{\substack{1\leqslant i\leqslant6\\i\neq j}}n_i,\qquad \ell=n_j,\qquad H=\prod_{\substack{1\leqslant i\leqslant6\\i\neq j}}N_i,\qquad L=N_j.
\end{equation*}
Thus, $\mathcal{T}^{*}$ is a sum of "Type II" satisfying $x^{1/4}\ll H\ll x^{1/2}$. By Lemma \ref{Type II est}, we have 
\begin{equation*}
\begin{aligned}
  x^{-\varepsilon}\cdot\mathcal{T}^{*}&\ll x^{3/4}{M_1}^{9/8}{M_2}^{7/8}+x^{7/8}M_1{M_2}^{3/4}+x^{13/16}{M_1}^{19/16}{M_2}^{3/4}+x^{15/16}(M_1M_2)^{3/4}\\
  &+x^{1/2}M_1M_2.
\end{aligned}
\end{equation*}
\subsubsection*{Case 3}
\indent If there exists an $N_j$ such that $x^{1/4}\leqslant N_j< x^{1/2}$, then we take
\begin{equation*}
  h=n_j,\qquad \ell=\prod_{\substack{1\leqslant i\leqslant6\\i\neq j}}n_i,\qquad H=N_j,\qquad L=\prod_{\substack{1\leqslant i\leqslant6\\i\neq j}}N_i.
\end{equation*}
Thus, $\mathcal{T}^{*}$ is a sum of "Type II" satisfying $x^{1/4}\ll H\ll x^{1/2}$. By Lemma \ref{Type II est}, we have 
\begin{equation*}
\begin{aligned}
  x^{-\varepsilon}\cdot\mathcal{T}^{*}&\ll x^{3/4}{M_1}^{9/8}{M_2}^{7/8}+x^{7/8}M_1{M_2}^{3/4}+x^{13/16}{M_1}^{19/16}{M_2}^{3/4}+x^{15/16}(M_1M_2)^{3/4}\\
  &+x^{1/2}M_1M_2.
\end{aligned}
\end{equation*}
\subsubsection*{Case 4} 
\indent If $N_j<x^{1/4}(j=1,2,3,4,5,6)$, without loss of generality, we assume that

\qquad\;\;$N_1\geqslant N_2\geqslant\cdots\geqslant N_6$. Let $r$ denote the natural number $j$ such that 
\begin{equation*}
  N_1N_2\cdots N_{j-1}<x^{1/4},\qquad N_1N_2\cdots N_j\geqslant x^{1/4}.
\end{equation*}
Since $N_1<x^{1/4}$ and $N_6<x^{1/4}$, then $2\leqslant r\leqslant5$. Thus we have
\begin{equation*}
  x^{1/4}\leqslant N_1N_2\cdots N_r=(N_1N_2\cdots N_{r-1})\cdot N_r<x^{1/4}\cdot x^{1/4}=x^{1/2}.
\end{equation*}
Let
\begin{equation*}
  h=\prod_{i=1}^{r}n_i,\qquad \ell=\prod_{i=r+1}^{6}n_i,\qquad H=\prod_{i=1}^{r}N_i,\qquad L=\prod_{i=r+1}^{6}N_i.
\end{equation*}
At this time, $\mathcal{T}^{*}$ is a sum of "Type II" satisfying $x^{1/4}\ll H\ll x^{1/2}$. By Lemma \ref{Type II est}, we have 
\begin{equation*}
\begin{aligned}
  x^{-\varepsilon}\cdot\mathcal{T}^{*}&\ll x^{3/4}{M_1}^{9/8}{M_2}^{7/8}+x^{7/8}M_1{M_2}^{3/4}+x^{13/16}{M_1}^{19/16}{M_2}^{3/4}+x^{15/16}(M_1M_2)^{3/4}\\
  &+x^{1/2}M_1M_2.
\end{aligned}
\end{equation*}
Combining the above four cases, we derive that 
\begin{equation*}
\begin{aligned}
  x^{-\varepsilon}\cdot\mathcal{T}^{*}&\ll x^{3/4}{M_1}^{9/8}{M_2}^{7/8}+x^{7/8}M_1{M_2}^{3/4}+x^{13/16}{M_1}^{19/16}{M_2}^{3/4}+x^{15/16}(M_1M_2)^{3/4}\\
  &+x^{1/2}{M_1}^{5/4}M_2,
\end{aligned}
\end{equation*}
which combined with (\ref{S_1}) and (\ref{S1*}) yields
\begin{equation*}
\begin{aligned}
  x^{-\varepsilon}\mathcal{S}_1&\ll x^{5/4}{M_1}^{3/8}{M_2}^{1/8}+x^{11/8}{M_1}^{1/4}+x^{21/16}{M_1}^{7/16}+x^{23/16}+x^{11/8}{M_1}^{3/16}+x{M_1}^{1/2}{M_2}^{1/4}.
\end{aligned}
\end{equation*}
Thus we obtain
\begin{equation}\label{R1 est}
\begin{aligned}
  x^{-\varepsilon}R_1\ll x^{5/4}{M_1}^{3/8}{M_2}^{1/8}+x^{11/8}{M_1}^{1/4}+x^{21/16}{M_1}^{7/16}+x^{23/16}+x^{11/8}{M_1}^{3/16}+x{M_1}^{1/2}{M_2}^{1/4}.
\end{aligned}
\end{equation}

\subsection{Estimate of \texorpdfstring{$R_2$}{}}
\
\newline \indent By the same argument on $R_1$, we derive that
\begin{equation}\label{R2}
  R_2\ll\frac{x^{1/2}}{(M_1M_2)^{3/4}}|{\mathcal{S}_2}^{*}|,
\end{equation}
where
\begin{equation}\label{S2*}
  {\mathcal{S}_2}^{*}=\sum_{\substack{m_1\sim M_1,m_2\sim M_2\\m_1\neq m_2}}d(m_1)d(m_2)
  \sum_{x<n\leqslant2x}\Lambda(n)\:\mathbf{e}(2(\sqrt{m_1}+\sqrt{m_2})\sqrt{n}).
\end{equation}
Using Lemma \ref{sum to int}, Lemma \ref{first derivative} and  the Cauchy-Schwarz's inequality, one has
\begin{equation*}
\begin{aligned}
  &\sum_{x<n\leqslant2x}\Lambda(n)\:\mathbf{e}(2(\sqrt{m_1}+\sqrt{m_2})\sqrt{n})\\
  &\ll\left(\sum_{x<n\leqslant2x}\Lambda^2(n)\right)^{1/2}\left(\sum_{x<n\leqslant2x}\:\mathbf{e}(4(\sqrt{m_1}+\sqrt{m_2})\sqrt{n})\right)^{1/2}\\
  &\ll x^{1/2+\varepsilon}\left(\int_{x}^{2x}\mathbf{e}(4(\sqrt{m_1}+\sqrt{m_2})\sqrt{t})dt\right)^{1/2}\\
  &\ll x^{3/4+\varepsilon}(m_1m_2)^{-1/8},
\end{aligned}
\end{equation*}
which combined with (\ref{S2*}) yields
\begin{equation}\label{R2 est}
\begin{aligned}
  R_2&\ll\frac{x^{5/4+\varepsilon}}{(M_1M_2)^{3/4}}\sum_{\substack{m_1\sim M_1,m_2\sim M_2\\m_1\neq m_2}}d(m_1)d(m_2)(m_1m_2)^{-1/8}\\
  &\ll x^{5/4+\varepsilon}(M_1M_2)^{1/8}.
\end{aligned}
\end{equation}
Above all, since $M_1,M_2$ run over $1\leqslant M_1,M_2\leqslant M$, $1\ll M\ll x$, by (\ref{S11 exp}), (\ref{R1 est}), (\ref{R2 est}), after eliminating lower order terms, we conclude that 
\begin{equation}\label{S11 est}
\begin{aligned}
  S_{11}&=\frac{1}{4\pi^2}\sum_{x<p\leqslant2x}p^{1/2}\left(\sum_{m=1}^{\infty}\frac{d^2(m)}{m^{3/2}}-\sum_{m>M}\frac{d^2(m)}{m^{3/2}}\right)+O(x^{\varepsilon}(x^{5/4}{M_1}^{3/8}{M_2}^{1/8}+x^{11/8}{M_1}^{1/4}\\
  &+x^{21/16}{M_1}^{7/16}+x^{23/16}+x^{11/8}{M_1}^{3/16}+x{M_1}^{1/2}{M_2}^{1/4}))\\
  &=\frac{\mathcal{C}}{4\pi^2}\sum_{x<p\leqslant2x}p^{1/2}+O(x^{\varepsilon}(x^{3/2}M^{-1/2}+x^{5/4}{M_1}^{3/8}{M_2}^{1/8}+x^{11/8}{M_1}^{1/4}+x^{21/16}{M_1}^{7/16}\\
  &+x^{23/16}+x^{11/8}{M_1}^{3/16}+x{M_1}^{1/2}{M_2}^{1/4})),\\
  &=\frac{\mathcal{C}}{4\pi^2}\sum_{x<p\leqslant2x}p^{1/2}+O(x^{\varepsilon}(x^{3/2}M^{-1/2}+x^{5/4}M^{1/2}+x^{21/16}{M}^{7/16}+x^{23/16}+x^{11/8}{M}^{3/16})),
\end{aligned}
\end{equation}
where we use the partial summation and $\mathcal{C}$ denotes $\sum_{m=1}^{\infty}\frac{d^2(m)}{m^{3/2}}$.

It remains to estimate $S_{12}$, suppose $M\ll x^{3/7}$, then by (\ref{S_22 est}), (\ref{S11 est}) and the Cauchy-Schwarz's inequality we have 
\begin{equation}\label{S12 est}
  S_{12}\ll (S_{11})^{1/2}(S_{22})^{1/2}\ll x^{3/2+\varepsilon}M^{-1/4}.
\end{equation}
Since $1\ll M\ll x^{3/7}$, combining (\ref{deltap2 exp}), (\ref{S_22 est}), (\ref{S11 est}), (\ref{S12 est}), taking $M=x^{1/4}$, we obtain
\begin{equation*}
  \sum_{x<p\leqslant2x}\Delta^2(p)=\frac{\mathcal{C}}{4\pi^2}\sum_{x<p\leqslant2x}p^{1/2}+O(x^{23/16+\varepsilon}).
\end{equation*}
Thus replacing $x$ by $x/2$, $x/2^2$, and so on, and adding up all the results, we obtain
\begin{equation}\label{delta2p est}
\begin{aligned}
  \sum_{p\leqslant x}\Delta^2(p)&=\frac{\mathcal{C}}{4\pi^2}\sum_{p\leqslant x}p^{1/2}+O(x^{23/16+\varepsilon}).
\end{aligned}
\end{equation}
Thus we have completed the proof of the Theorem.
\section*{Acknowledgement}

The authors would like to appreciate the referee for his/her patience in refereeing this paper.
This work is supported by the Beijing Natural Science Foundation(Grant No.1242003), and the National Natural Science Foundation of China (Grant No.11971476).

\bibliographystyle{plain}
\bibliography{reference.bib}

\begin{thebibliography}{10}

\bibitem{MR1544568}
Harald Cram\'{e}r.
\newblock \"{U}ber zwei {S}\"{a}tze des {H}errn {G}. {H}. {H}ardy.
\newblock {\em Math. Z.}, 15(1):201--210, 1922.

\bibitem{MR1027058}
\'{E}tienne Fouvry and Henryk Iwaniec.
\newblock Exponential sums with monomials.
\newblock {\em J. Number Theory}, 33(3):311--333, 1989.

\bibitem{MR2176479}
Jun Furuya.
\newblock On the average orders of the error term in the {D}irichlet divisor
  problem.
\newblock {\em J. Number Theory}, 115(1):1--26, 2005.

\bibitem{MR1576556}
G.~H. Hardy.
\newblock The {A}verage {O}rder of the {A}rithmetical {F}unctions {P}(x) and
  delta(x).
\newblock {\em Proc. London Math. Soc. (2)}, 15:192--213, 1916.

\bibitem{MR0678676}
D.~R. Heath-Brown.
\newblock Prime numbers in short intervals and a generalized {V}aughan
  identity.
\newblock {\em Canadian J. Math.}, 34(6):1365--1377, 1982.

\bibitem{MR1295953}
D.~R. Heath-Brown and K.~Tsang.
\newblock Sign changes of {$E(T)$}, {$\Delta(x)$}, and {$P(x)$}.
\newblock {\em J. Number Theory}, 49(1):73--83, 1994.

\bibitem{MR1199067}
M.~N. Huxley.
\newblock Exponential sums and lattice points. {II}.
\newblock {\em Proc. London Math. Soc. (3)}, 66(2):279--301, 1993.

\bibitem{MR2005876}
M.~N. Huxley.
\newblock Exponential sums and lattice points. {III}.
\newblock {\em Proc. London Math. Soc. (3)}, 87(3):591--609, 2003.

\bibitem{MR0792089}
Aleksandar Ivi\'{c}.
\newblock {\em The {R}iemann zeta-function}.
\newblock A Wiley-Interscience Publication. John Wiley \& Sons, Inc., New York,
  1985.
\newblock The theory of the Riemann zeta-function with applications.

\bibitem{MR2061214}
Henryk Iwaniec and Emmanuel Kowalski.
\newblock {\em Analytic number theory}, volume~53 of {\em American Mathematical
  Society Colloquium Publications}.
\newblock American Mathematical Society, Providence, RI, 2004.

\bibitem{MR0883536}
Henryk Iwaniec and Andr\'{a}s S\'{a}rk\"{o}zy.
\newblock On a multiplicative hybrid problem.
\newblock {\em J. Number Theory}, 26(1):89--95, 1987.

\bibitem{MR1215269}
Anatolij~A. Karatsuba.
\newblock {\em Basic analytic number theory}.
\newblock Springer-Verlag, Berlin, russian edition, 1993.

\bibitem{MR0644943}
G.~Kolesnik.
\newblock On the order of {$\zeta ({\frac{1}{2}}+it)$} and {$\Delta (R)$}.
\newblock {\em Pacific J. Math.}, 98(1):107--122, 1982.

\bibitem{MR2475967}
Yuk-Kam Lau and Kai-Man Tsang.
\newblock On the mean square formula of the error term in the {D}irichlet
  divisor problem.
\newblock {\em Math. Proc. Cambridge Philos. Soc.}, 146(2):277--287, 2009.

\bibitem{MR0930552}
Emmanuel Preissmann.
\newblock Sur la moyenne quadratique du terme de reste du probl\`eme du cercle.
\newblock {\em C. R. Acad. Sci. Paris S\'{e}r. I Math.}, 306(4):151--154, 1988.

\bibitem{MR2212877}
O.~Robert and P.~Sargos.
\newblock Three-dimensional exponential sums with monomials.
\newblock {\em J. Reine Angew. Math.}, 591:1--20, 2006.

\bibitem{MR0098718}
Kwang-Chang Tong.
\newblock On divisor problems. {II}, {III}.
\newblock {\em Acta Math. Sinica}, 6:139--152, 515--541, 1956.

\bibitem{MR1162488}
Kai~Man Tsang.
\newblock Higher-power moments of {$\Delta(x),\;E(t)$} and {$P(x)$}.
\newblock {\em Proc. London Math. Soc. (3)}, 65(1):65--84, 1992.

\bibitem{MR1512430}
J.~G. van~der Corput.
\newblock Zum {T}eilerproblem.
\newblock {\em Math. Ann.}, 98(1):697--716, 1928.

\bibitem{MR1509041}
Georges Vorono\"{\i}.
\newblock Sur une fonction transcendante et ses applications \`a la sommation
  de quelques s\'{e}ries.
\newblock {\em Ann. Sci. \'{E}cole Norm. Sup. (3)}, 21:207--267, 1904.

\bibitem{MR2067871}
Wenguang Zhai.
\newblock On higher-power moments of {$\Delta(x)$}. {II}.
\newblock {\em Acta Arith.}, 114(1):35--54, 2004.

\bibitem{MR2168766}
Wenguang Zhai.
\newblock On higher-power moments of {$\Delta(x)$}. {III}.
\newblock {\em Acta Arith.}, 118(3):263--281, 2005.

\end{thebibliography}

\end{document}